\documentclass{amsart}
\usepackage{amsthm, amssymb, amsfonts}
\usepackage{lmodern}
\usepackage{color}
\usepackage[T5]{fontenc}
\usepackage{amscd,amssymb}
\usepackage[v2,cmtip]{xy}
\usepackage{mathrsfs}

\title[On the module structure over the Steenrod algebra]{On the module structure over the Steenrod algebra of the Dickson algebra}
  
\author{Nguyen Sum}

\theoremstyle{plain}
\newtheorem{thm}{Theorem}[section]
\newtheorem{prop}[thm]{Proposition}
\newtheorem{lem}[thm]{Lemma}
\newtheorem{corl}[thm]{Corollary}
\theoremstyle{definition}
\newtheorem{defn}[thm]{Definition}

\begin{document} 
\setlength{\baselineskip}{14pt}

\begin{abstract} Let $p$ be an odd prime number. We study the problem of determining the module structure over the mod $p$ Steenrod algebra $\mathcal A(p)$ of the Dickson algebra $D_n$ consisting of all modular invariants of general linear group $GL(n,\mathbb F_p)$. Here $\mathbb F_p$ denotes the prime field of $p$ elements. In this paper, we give an explicit answer for $n=2$. More precisely, we explicitly compute the action of the Steenrod-Milnor operations $St^{S,R}$ on the generators of $D_n$ for $n=2$ and for either $S=\emptyset, R=(i)$ or $S=(s), R=(i)$ with $s,i$ arbitrary nonnegative integers.
\end{abstract}

\maketitle

\section{Introduction}

  Let $p$ be an odd prime number and let $GL_n  =  GL(n,\mathbb F_p)$ be the general linear group over the prime field $\mathbb F_p$ of $p$ elements. This group acts naturally on  the algebra $P_n := E(x_1,x_2,\ldots, x_n)\otimes P(y_1,y_2,\ldots , y_n)$. Here and in what follows, $E(.,.,\ldots)$ and $P(.,., \ldots)$ are the exterior and polynomial algebras over $\mathbb F_p$ generated by the  indicated variables. We grade $P_n$ by assigning $\deg x_i=1$ and $\deg y_i = 2.$

Dickson showed in \cite{di} that the invariant algebra  $P(y_1,y_2,\ldots , y_n)^{GL_n}$ is a polynomial algebra over $\mathbb F_p$  generated by the Dickson invariants $Q_{n,s},\ 0\leqslant s<n$. In \cite{m1}, Huynh Mui proved that the Dickson algebra $D_n=P_n^{GL_n}$ of invariants  is generated by the Dickson invariants $Q_{n, s},\ 0 \leqslant  s <  n,$ and Mui invariants $R_{n; s_1,  \ldots, s_k},\  0 \leqslant s_1 < \ldots < s_k < n.$ 

It is well known that  $P_n$ is a module over the Steenrod algebra $\mathcal A(p)$. The action of   $\mathcal{A}(p)$ on $P_n$ is determined by the formulas 
\begin{align*}
\beta x_j &= y_j,\ \beta y_j = 0,\\
P^i(x_j) &= \begin{cases} x_j, &i=0,\\ 0, &i>0,\end{cases} \ \ 
P^i(y_j) = \begin{cases} y_j, & i=0,\\ y_j^p,&i=1,\\ 0, &i>1,\end{cases}
\end{align*}
and subject to the Cartan formulas
\begin{align*}
\beta(xy) &= \beta(x)y +(-1)^{\deg x} x\beta(y),\\
P^r(xy) &= \sum_{i=0}^rP^i(x)P^{r-i}(y),
\end{align*}
for $x,y \in P_n$ and $\beta$ is the Bockstein homomorphism (see Steenrod \cite{ss}).

 Since this action commutes with  the one of $GL_n$, it induces  an action of $\mathcal{A}(p)$ on Dickson algebra $D_n$. So $D_n$ is a submodule of $P_n$. Note that the polynomial algebra $P(y_1,y_2,\ldots,y_n)$ is a submodule of $P_n$ and $P(y_1,y_2,\ldots,y_n)^{GL_n}$ is a submodule of the algebra $D_n$.

Let $\tau_s, \xi_i$ be the Milnor elements of degrees $2p^s-1,\
2p^i-2$ respectively in the dual algebra $\mathcal{A}(p)^*$ of $\mathcal{A}(p)$.
In \cite{mi}, Milnor showed that as an algebra,
$$\mathcal{A}(p)^* = E(\tau_0,\tau_1,\ldots ) \otimes P(\xi_1,\xi_2,\ldots ). $$
Then $\mathcal{A}(p)^*$ has a basis consisting of all monomials
    $$\tau_S\xi^R \ =\ \tau_{s_1}\ldots \tau_{s_k}\xi_1^{r_1}\ldots
     \xi_m^{r_m},$$
with $S = (s_1,\ldots ,s_k),\ 0 \leqslant s_1 <\ldots <s_k ,
 R = (r_1,\ldots ,r_m),\ r_i \geqslant 0 $. Let $St^{S,R} \in \mathcal{A}(p)$
denote the dual of $\tau_S\xi^R$ with respect to that basis.
Then $\mathcal{A}(p)$ has a basis consisting of all operations $St^{S,R}$. For $S=\emptyset, R=(r)$, $St^{\emptyset, (r)}$ is nothing but the Steenrod operation $P^r$. So, we call $St^{S,R}$ the Steenrod-Milnor operation of type $(S,R)$.  

The operations $St^{S,R}$ have the following fundamental properties:

-- $St^{\emptyset,(0)}=1, \ St^{(0),(0)} = \beta.$

-- $St^{S,R}(z) = 0$ if $z\in P_n$ and $\deg z < k + 2(r_1+r_2+\ldots+r_m)$.

-- The Cartan formula
$$St^{S,R}(zt) =\sum_{\overset{\scriptstyle{S_1\cup S_2=S}}{R_1+R_2=R}}(-1)^{(\deg z+\ell(S_1))\ell(S_2)}(S:S_1,S_2)St^{S_1,R_1}(z)St^{S_2,R_2}(t),$$
where $ R_1=(r_{1i}),\ R_2=(r_{2i}),\ R_1+R_2=(r_{1i}+r_{2i}), S_1\cap S_2=\emptyset, z,t\in P_n$,  $\ell(S_j)$ means the length of $S_j$ and $$(S:S_1,S_2)= \text{sign}\begin{pmatrix} s_1&\ldots &s_h& s_{h+1}&\ldots &s_k\\
s_{1,1}&\ldots &s_{1,h}&s_{2,1}&\ldots &s_{2,k-h}\end{pmatrix} ,$$ 
with $S_1=(s_{1,1},\ldots,s_{1,h}), s_{1,1}<\ldots <s_{1,h}$, $S_2=(s_{2,1},\ldots,s_{2,k-h}), s_{2,1} < \ldots < s_{2,k-h}$ (see Mui \cite{m2}).

The action of $St^{S,R}$ on Dickson invariants $Q_{n,s}$  has partially been studied by many authors. This action for $S=\emptyset,\ \! R=(i)$ was explicitly determined by Madsen-Milgram \cite{mm},  Smith-Switzer \cite{ss},  Hung-Minh \cite{hm}, Kechagias \cite{ke}, Sum \cite{s2}, Wilkerson \cite{wi}. This action for either $S=(s), R=(0)$ or $S=\emptyset, R=(0,\ldots,0,1)$ with 1 at the $i$-th place, was studied by Wilkerson \cite{wi}, Neusel \cite{ne}, Sum \cite{s3}.

In this paper, we explicitly determine the action of the  Steenrod-Milnor operations $St^{S,R}$ on Dickson invariants $Q_{2,0}, Q_{2,1}$ and Mui invariants $R_{2;0}, R_{2;1}, R_{2;0,1}$ for either $S=\emptyset, R=(i)$ or $S=(s), R=(i)$.

In Section 2 we recall some results on the modular invariants of the general linear group $GL_2$ and the action of the Steenrod-Milnor operations on the generators of $P_2$. In Section 3, we compute the action of the Steenrod operations on Dickson-Mui invariants. Finally, in Section 4, we explicitly determine the action of the Steenrod-Milnor operations $St^{(s),(i)}$ on $Q_{2,0}, Q_{2,1}, R_{2;0}, R_{2;1}$ and $R_{2;0,1}$.

\section{Preliminaries}

\begin{defn} 
Let $u,v$ be  nonnegative integers.  Following Dickson \cite{di}, Mui \cite{m1}, we define  
$$ [u;v]  = 
\vmatrix y_1^{p^u}&y_2^{p^u}\\
   y_1^{p^v} &  y_2^{p^v}
   \endvmatrix,\ 
[1;u]  = 
\vmatrix x_1&x_2\\
   y_1^{p^u} &  y_2^{p^u}
   \endvmatrix.$$
\end{defn}

In particular, we set
 \begin{align*} L_2 &= [0,1],\ L_{2,0}=[1,2],\ L_{2,1} = [0,2],\\ 
M_{2;0} &= [1;1],\ M_{2;1} = [1;0],\ M_{2;0,1} = x_1x_2.
\end{align*}

 The polynomial $[u,v]$ is divisible by $L_2$. Then, Dickson invariants $Q_{2,0}, Q_{2,1}$ and Mui invariants $ R_{2;0}, R_{2;1}, R_{2;0,1}$ are defined by
\begin{align*} Q_{2,0}&=L_{2,0}/L_2,\ Q_{2,1} = L_{2,1}/L_2,\\
R_{2;0} &= M_{2;0}L_2^{p-2},\ R_{2;1} = M_{2;1}L_2^{p-2},\ R_{2;0,1} = M_{2;0,1}L_2^{p-2}.
\end{align*}

Now we prepare some data in order to prove our main results. 
First, we recall the following which will be needed in the next sections.

Let $\alpha_i(a)$ denote the $i$-th coefficient in $p$-adic expansion of a non-negative integer $a$. That means
$$a= \alpha_0(a)p^0+\alpha_1(a)p^1+\alpha_2(a)p^2+ \ldots ,$$ 
for $0 \leqslant \alpha_i(a) <p, i\geqslant 0.$ 

Denote by $I(u,v)$ the set of all integers $a$ satisfying
\begin{align*}& \alpha_i(a)+\alpha_{i+1}(a) \leqslant 1,\text{ for any }  i,\\
& \alpha_i(a)=0, \text{ for either } i<u \text{ or } i\geqslant v-2.
\end{align*}

\begin{prop}[Sum \cite{s3}] \label{md2.9}Under the above notations, we have
$$ [u,v] = \sum_{a\in I(u,v)}(-1)^aL_2^{p^u+p(p-1)a}Q_{2,1}^{\frac{p^{v-1}-p^u}{p-1}-(p+1)a}.$$
\end{prop}

\begin{lem}[Sum \cite{s2}]\label{bd2.2} Let $b$ be a nonnegative integer and $\varepsilon = 0, 1$. We  have
$$ St^{S,R}(x_k^\varepsilon y_\ell^b) = \begin{cases} \binom bR x_k^\varepsilon y_\ell^{b + \vert R\vert}\ , &S = \emptyset\ ,\\
                              \varepsilon \binom bRy_k^{p^s}y_\ell^{b + \vert R\vert}\ , &S = (s), \ s \geqslant 0 ,\\
                              0\ , &\text{otherwise.} \end{cases}$$
Here $\binom{b}{R} = \frac{b!}{(b-r_1-r_2-\ldots-r_m)!r_1!\ldots r_m!}$ for $r_1+r_2+\ldots+r_m\leqslant b$ and $\binom{b}{R}=0$ for $r_1+r_2+\ldots+r_m > b$ and $|R|= (p-1)r_1+(p^2-1)r_2+\ldots + (p^m-1)r_m$. 
\end{lem}
Note that for $R=(i), \binom bR= \binom bi$ is the binomial coefficient. By convention, we set $\binom bi=0$ for $i<0$.

Applying Lemma \ref{bd2.2} to $P^i=St^{\emptyset,(i)}$, we get
\begin{corl}[Steenrod \cite{st}] Let $b, i$ be nonnegative integers. Then we have
$$P^iy_k^b =  \binom niy_k^{b+(p-1)i}.$$
\end{corl}
Since $\binom {p^e}i = 0$ in $\mathbb F_p$ for $1 <i<p^e$, we get
\begin{corl}\label{hq2.4} For any nonnegative integers $e, i$,
$$P^iy_k^{p^e} = \begin{cases} y_k^{p^e}, & i=0,\\
y_k^{p^{e+1}}, & i=p^e,\\
0, &\text{otherwise}.
\end{cases} $$
\end{corl}
Applying Corollary \ref{hq2.4} and the Cartan formula to $[u,v]=y_1^{p^u}y_2^{p^v} - y_1^{p^v}y_2^{p^u}$, we obtain
\begin{lem}[Mui \cite{m1}]\label{bd2.5} Let $u,v, i$ be nonnegative integers. Then we have
$$P^i{[u,v]} = \begin{cases} {[u,v]}, & i=0,\\
{[u+1,v]}, & i=p^u,\\
{[u,v+1]}, & i=p^v,\\
{[u+1,v+1]}, & i=p^u+p^v,\\
0, &\text{otherwise}.
\end{cases} $$
\end{lem}

Since $L_2={[0,1]}, L_{2,0} = {[1,2]}, L_{2,1} ={[0,2]}$, from Lemma \ref{bd2.5}, we get 
\begin{corl}\label{hq2.6}  For any nonnegative integers $i$,
\begin{align*}P^iL_2 &= \begin{cases} L_2, &i=0,\\
L_2Q_{2,1}, &i= p,\\
L_2Q_{2,0}, &i= p+1,\\
0, &\text{otherwise,}\end{cases}\\
P^iL_{2,0} &= \begin{cases} {[0,1]}=L_2Q_{2,0}, &i=0,\\
{[1,3]}= L_2Q_{2,0}Q_{2,1}^p, &i= p^2,\\
{[2,3]} = L_2Q_{2,0}^{p+1}, &i= p^2+p,\\
0, &\text{otherwise,}\end{cases}\\
P^iL_{2,1} &= \begin{cases} {[0,2]}=L_2Q_{2,1}, &i=0,\\
{[1,2]} = L_2Q_{2,0}, &i= 1,\\
{[0,3]} = L_2(Q_{2,1}^{p+1}-Q_{2,0}^p), &i=p^2,\\
{[1,3]} = L_2Q_{2,0}Q_{2,1}^{p}, &i= p^2+1,\\
0, &\text{otherwise.}\end{cases}
\end{align*} 
\end{corl}
Combining Lemma \ref{bd2.2}, Corollary \ref{hq2.4} and the Cartan formula gives
\begin{lem}\label{bd2.7} Let $s, i$ be nonnegative integers. Then we have
$$St^{(s),(i)}{[1;u]} = \begin{cases} {[s,u]}, &i=0,\\ {[s,u+1]}, &i=p^u,\\ 0, &\text{otherwise}.
\end{cases}$$
\end{lem}
Applying Lemma \ref{bd2.2} and Corollary \ref{hq2.4} to $M_{2;0}={[1;1]},\ M_{2;1} = {[1;0]}$, we obtain
\begin{corl}\label{hq2.8}  For any nonnegative integer $i$,
\begin{align*}P^i(M_{2;0}) &= \begin{cases} {[1;1]}=M_{2;0}, &i=0,\\
{[1;2]}= M_{2;0}Q_{2,1}-M_{2;1}Q_{2,0}, &i = p,\\
0, &\text{otherwise,}\end{cases}\\
P^i(M_{2;1}) &= \begin{cases} {[1;0]}=M_{2;1}, &i=0,\\
{[1;1]}= M_{2;0}, &i = 1,\\
0, &\text{otherwise.}\end{cases}
\end{align*}
\end{corl}

\section{The action of the Steenrod operations on Dickson-Mui invariants}

First of all, we prove the following which was proved in Hung-Minh \cite{hm}, by another method.
\begin{thm}[Hung-Minh \cite{hm}] \label{dl3.1} For any nonnegative integer $i$ and $s=0,1$, we have
$$P^iQ_{2,s} = \begin{cases} (-1)^k\binom {k+s}rQ_{2,0}^{r+1-s}Q_{2,1}^{k+s-r},& i=kp+r, 0\leqslant r-s \leqslant k < p,\\
0, &\text{otherwise.}\end{cases}$$
\end{thm}
\begin{proof} Recall that $\deg Q_{2,s} = 2(p^2-p^s) < 2p^2$. Hence, $P^iQ_{2,s} =0$ for $i \geqslant p^2$. Suppose that $i < p^2$. Then,  using the $p$-adic expansion of $i$, we have
$$i=kp+r\ \text{ for}\ 0\leqslant  k,\ r<p .$$

 We prove the theorem by  induction on $k$. We have $P^0Q_{2,s} = Q_{2,s}$. According to Corollary \ref{hq2.6}, 
\begin{align*}0&=P^1L_{2,0} = P^1(L_2Q_{2,0}) = L_2P^1Q_{2,0},\\ L_2Q_{2,0}&=P^1L_{2,1} = P^1(L_2Q_{2,1}) = L_2P^1Q_{2,1}.\end{align*} 
These equalities imply $P^1Q_{2,0} = 0, P^1Q_{2,1} = Q_{2,0}$. For $1< r<p$, $P^rL_{2,s} = 0$  and $P^rL_2 =0$. Using the Cartan formula and Corollary \ref{hq2.6}, we have
$$ 0=P^rL_{2,s} = P^r(L_2Q_{2,s}) = L_2P^rQ_{2,s}.$$
This implies $P^rQ_{2,s} = 0$. So the theorem holds for $k=0$ and any $0\leqslant r <p$. Suppose $0<k<p$ and the theorem is true for $k-1$ and any $0\leqslant r <p$. Using the Cartan formula, Corollary \ref{hq2.6} and the inductive hypothesis, we get
\begin{align*}
0 &= P^iL_{2,s} = P^i(L_2Q_{2,s})\\ 
&= L_2P^iQ_{2,s} + L_2Q_{2,1}P^{i-p}Q_{2,s}+L_2Q_{2,0}P^{i-p-1}Q_{2,0}\\
&=L_2P^iQ_{2,s} + L_2Q_{2,1}(-1)^{k-1}\binom{k-1+s}rQ_{2,0}^{r+1-s}Q_{2,1}^{k+s-r-1}\\
&\hskip4cm+L_2Q_{2,0}(-1)^{k-1}\binom{k-1+s}{r-1}Q_{2,0}^{r-s}Q_{2,1}^{k+s-r}\\
&=L_2P^iQ_{2,s}+ (-1)^{k-1}\Big(\binom{k-1+s}r +\binom{k-1+s}{r-1}\Big)L_2Q_{2,0}^{r+1-s}Q_{2,1}^{k+s-r}.
\end{align*}

From this equality and the relation $\binom{k-1+s}r +\binom{k-1+s}{r-1}=\binom {k+s}r$, we see that the theorem holds for $k$. The proof is completed.
\end{proof}

To compute the action of $P^i$ on $R_{2;0}, R_{2;1}, R_{2;0,1}$ we need the following
\begin{lem}\label{bd3.2} Let $i$ be a nonnegative integer. Then we have
$$P^iL_2^{p-2} = \begin{cases} (-1)^k(k+1)\binom kr L_2^{p-2}Q_{2,0}^rQ_{2,1}^{k-r}, &i=kp+r, 0\leqslant r \leqslant k < p,\\
0, &\text{otherwise}.
\end{cases}$$
\end{lem} 
\begin{proof} Note that $\deg L_2^{p-2} = 2(p-2)(p+1) <2p^2$. So $P^iL_2^{p-2} = 0$ for $i \geqslant p^2$. Hence, it suffices to prove the theorem for $i=kp+r$ with $0\leqslant k,r < p$.

 Since $P^0 = 1$, we have $P^0L_2^{p-2} = L_2^{p-2}$. If $0< r <p$ then from Theorem \ref{dl3.1}, the Cartan formula and the relation  $Q_{2,0} = L_2^{p-1}= L_2L_2^{p-2}$, we get 
$$0= P^rQ_{2,0} = L_2P^rL_2^{p-2}.$$
 This implies $P^rL_2^{p-2} =0$. The lemma is true for $k=0$ and $0\leqslant r <p$.

Suppose that $0<k<p$ and the lemma holds for $k-1$ and any $0\leqslant r <p$. Using the Cartan formula, Theorem \ref{dl3.1}, Corollary \ref{hq2.6} and the inductive hypothesis, we have
\begin{align*} (-1)^k\binom kr Q_{2,0}^{r+1}Q_{2,1}^{k-r} &= P^iQ_{2,0} = P^i(L_2L_2^{p-2})\\
&= L_2P^iL_2^{p-2} + L_2Q_{2,1}P^{i-p}L_2^{p-2}+ L_2Q_{2,0}P^{i-p-1}L_2^{p-2}\\ 
&=L_2P^iL_2^{p-2} + L_2Q_{2,1}(-1)^{k-1}k\binom{k-1}rL_2^{p-2}Q_{2,0}^{r}Q_{2,1}^{k-r-1}\\
&\hskip2cm+ L_2Q_{2,0}(-1)^{k-1}k\binom{k-1}{r-1}L_2^{p-2}Q_{2,0}^{r-1}Q_{2,1}^{k-r}\\
&= L_2P^iL_2^{p-2} + (-1)^{k-1}k\Big(\binom{k-1}r +\binom{k-1}{r-1}\Big)Q_{2,0}^{r+1}Q_{2,1}^{k-r}.
\end{align*}
This equality and the relation $\binom{k-1}r+\binom{k-1}{r-1}=\binom kr$ imply the lemma for $k$ and any $0\leqslant r <p$.
\end{proof}
\begin{thm}\label{dl3.3} Let $i$ be a nonnegative integer. We have 
$$ P^iR_{2;0} =\begin{cases}  (-1)^k\big((r+1)\binom krR_{2;0}Q_{2,0}^{r}Q_{2,1}^{k-r}
+k\binom{k-1}rR_{2;1}Q_{2,0}^{r+1}Q_{2,1}^{k-r-1}\big),\\
\hskip5cm i=kp+r, 0\leqslant r \leqslant k < p,\\
0,  \hskip4.6cm\text{otherwise}.
\end{cases}$$
\end{thm}
\begin{proof} Note that  $\deg R_{2;0} = 2p^2-3 < 2p^2$. So $P^iR_{2;0} = 0$ for $i \geqslant p^2$. We prove the theorem for $i=kp+r$ with $0 \leqslant k,r <p$. 

For $k=r=0$, $P^0R_{2;0} = R_{2;0}$. For $k= 0, 0<r<p$, applying the Cartan formula and Corollary \ref{hq2.8}, we get
$$ P^rR_{2;0} = P^r(M_{2;0}L_2^{p-2}) = M_{2;0}P^rL_2^{p-2}=0.$$
The theorem holds for $k=0$ and $0\leqslant r <p$. 

Suppose that $0<k < p$. Using the Cartan formula, Corollary \ref{hq2.8} and Lemma \ref{bd3.2}, we obtain
\begin{align*}
P^iR_{2;0} &= P^i(M_{2;0}L_2^{p-2}) = P^0M_{2;0}P^iL_2^{p-2} + P^pM_{2;0}P^{i-p}L_2^{p-2}\\
&= M_{2;0}P^iL_2^{p-2} + (M_{2;0}Q_{2,1} - M_{2;1}Q_{2,0})P^{i-p}L_2^{p-2}\\
&= M_{2;0}(-1)^k(k+1)\binom kr L_2^{p-2}Q_{2,0}^rQ_{2,1}^{k-r}\\
&\quad + (M_{2;0}Q_{2,1} - M_{2;1}Q_{2,0})(-1)^{k-1}k\binom {k-1}rL_2^{p-2}Q_{2,0}^rQ_{2,1}^{k-r-1}\\
&= (-1)^k\Big(\big((k+1)\binom kr - k\binom{k-1}r\big)R_{2;0}Q_{2,0}^rQ_{2,1}^{k-r}\\ &\hskip5cm + k\binom{k-1}rR_{2;1}Q_{2,0}^{r+1}Q_{2,1}^{k-r-1}\Big).
\end{align*}

This equality and the relation $(k+1)\binom kr - k\binom{k-1}r = (r+1)\binom kr$ imply the theorem for $k$ and $0\leqslant r <p$.
\end{proof}

By an analogous argument as given in the proof of Theorem \ref{dl3.3}, we can easily obtain the following
\begin{thm} For any nonnegative integer $i$, we have
\begin{align*}P^iR_{2;1} &=\begin{cases}  (-1)^k(k+1)\big(\binom krR_{2;1}Q_{2,0}^{r}Q_{2,1}^{k-r}
+\binom{k}{r-1}R_{2;0}Q_{2,0}^{r-1}Q_{2,1}^{k-r+1}\big),\\
\hskip5cm i=kp+r, 0\leqslant r \leqslant k < p,\\
0,  \hskip4.6cm\text{otherwise},
\end{cases}\\
P^iR_{2;0,1} &= \begin{cases} (-1)^k(k+1)\binom krR_{2;0,1}Q_{2,0}^rQ_{2,1}^{k-r}, &i=kp+r, 0\leqslant r \leqslant k <p,\\
0, &\text{otherwise}.
\end{cases}
\end{align*}
\end{thm}
\section{On the action of the Steenrod-Milnor operations\\ on Dickson-Mui invariants}
 In this section, we compute the action of $St^{(s),(i)}$ on Dickson-Mui invariants. It is easy to see that $St^{(s),(i)}Q_{2,s} = 0$. So we need only to compute the action of $St^{(s),(i)}$ on $R_{2;0}, R_{2;1}$ and $R_{2;0,1}$.

First, we recall the following
\begin{lem}[Sum \cite{s1}]\label{bd4.1}  For any nonnegative integers $s, i$,
\begin{align*}St^{(s),(i)}(M_{2;0}) &= \begin{cases} {[s,1]}, &i=0,\\
{[s,2]}, &i = p,\\
0, &\text{otherwise,}\end{cases}\ \
St^{(s),(i)}(M_{2;1}) = \begin{cases} {[s,0]}, &i=0,\\
{[s,1]}, &i = 1,\\
0, &\text{otherwise.}\end{cases}\end{align*}
\end{lem}

This lemma can easily be proved by using the Cartan formula and Lemma \ref{bd2.2}.

\begin{thm}\label{dl4.2} Let $s, i$ be nonnegative integers. Then we have
$$St^{(s),(i)}R_{2;0} =\begin{cases}  (-1)^k(r+1)\binom krQ_{2,0}^{k+1}Q_{2,1}^{k-r}, \ \ \ \ s=0,  i= kp+r, 0\leqslant r \leqslant k < p,\\
(-1)^{k+1}k\binom{k-1}{r}Q_{2,0}^{r+2}Q_{2,1}^{k-1-r}, \  \ \ s=1, i=kp+r, 0\leqslant r < k < p,\\
(-1)^{k+1}(k+1)\binom{k}{r}Q_{2,0}^{r+2}Q_{2,1}^{k-r},\   s=2, i=kp+r, 0\leqslant r \leqslant k < p,\\ 
\displaystyle{(-1)^{k+1}\binom kr\Big((k+1)\sum_{a\in I(1,s)}(-1)^aQ_{2,0}^{pa+r+2}Q_{2,1}^{\frac{p^{s-1}-p}{p-1}-(p+1)a+k-r}}\\
\qquad-\displaystyle{(k-r)\sum_{a\in I(2,s)}(-1)^aQ_{2,0}^{p(a+1)+r+2}Q_{2,1}^{\frac{p^{s-1}-p^2}{p-1}-(p+1)a+k-r-1}}\Big),\\
\hskip5.4cm s >2, i =kp+r,  0\leqslant r \leqslant k < p,\\ 
0, \hskip5cm \text{otherwise}.
\end{cases}$$
\end{thm}
\begin{proof} Since $\deg R_{2;0} = 2p^2- 3$, $St^{(s),(i)}R_{2;0} = 0$ for $i\geqslant p^2$. Suppose $i<p^2$, then using the $p$-adic expansion of $i$, we have
$i=kp+r\ \text{ for}\ 0\leqslant  k,\ r<p .$
 
We have $St^{(s),(0)}R_{2;0} = St^{(s),(0)}(M_{2;0}L_2^{p-2}) = St^{(s),(0)}M_{2;0}P^0L_2^{p-2}=[s,1]L_2^{p-2}$. For $0<r<p$, using the Cartan formula, Lemma \ref{bd3.2} and Lemma \ref{bd4.1}, we get
$$St^{(s),(r)}R_{2;0} = St^{(s),(r)}(M_{2;0}L_2^{p-2}) =St^{(s),(0)}(M_{2;0})P^rL_2^{p-2}= 0.$$
The above equalities and Proposition \ref{md2.9} imply the theorem for $k=0$.

For $0<k<p$, using the Cartan formula, Lemma \ref{bd3.2} and Lemma \ref{bd4.1} we obtain
\begin{align*}
St^{(s),(i)}R_{2;0} &= St^{(s),(i)}(M_{2;0}L_2^{p-2})\\
&= St^{(s),(0)}M_{2;0}P^iL_2^{p-2} + St^{(s),(p)}M_{2;0}P^{i-p}L_2^{p-2}\\
&= [s,1](-1)^k(k+1)\binom krL_2^{p-2}Q_{2,0}^rQ_{2,1}^{k-r}\\
&\quad + [s,2](-1)^{k-1}k\binom{k-1}rL_2^{p-2}Q_{2,0}^rQ_{2,1}^{k-r-1}.
\end{align*}

Combining this equality and Proposition \ref{md2.9}, we obtain the theorem.
\end{proof}
\begin{thm} For any nonnegative integers $s, i$,
$$St^{(s),(i)}R_{2;1} =\begin{cases}  (-1)^k(k+1)\binom k{r-1}Q_{2,0}^{r}Q_{2,1}^{k-r+1}, \ \ s=0,  i= kp+r, 0\leqslant r \leqslant k < p,\\
(-1)^{k+1}(k+1)\binom{k}{r}Q_{2,0}^{r+1}Q_{2,1}^{k-r}, \ \ \ s=1, i=kp+r, 0\leqslant r \leqslant k < p,\\
\displaystyle{(-1)^{k+1}(k+1)\Big(\binom kr\sum_{a\in I(0,s)}(-1)^aQ_{2,0}^{pa+r+1}Q_{2,1}^{\frac{p^{s-1}-1}{p-1}-(p+1)a+k-r}}\\
\qquad+\displaystyle{\binom{k}{r-1}\sum_{a\in I(1,s)}(-1)^aQ_{2,0}^{pa+r+1}Q_{2,1}^{\frac{p^{s-1}-p}{p-1}-(p+1)a+k-r+1}}\Big),\\
\hskip5.6cm s >1, i =kp+r,  0\leqslant r \leqslant k < p,\\ 
0, \hskip5.2cm \text{otherwise}.
\end{cases}$$
\end{thm}
\begin{proof} Since $\deg R_{2;1} = 2(p^2-p) - 1$, $St^{(s),(i)}R_{2;1}=0$ for $i\geqslant p^2$. Suppose $ i < p^2$ and $i=kp+r$ with $0\leqslant k,r <p$. Using the Cartan formula and Lemma \ref{bd4.1}, we have 
$$ St^{(s), (0)}R_{2;1} = St^{(s), (0)}(M_{2;1})L_2^{p-2}={[s,0]}L_2^{p-2}.$$
From this and Proposition \ref{md2.9}, we see that the theorem is true for $i=0$.

For $i>0$, using the Cartan formula, Lemma \ref{bd4.1} and Lemma \ref{bd3.2}, we obtain
\begin{align*}
 St^{(s), (i)}R_{2;1} &= St^{(s), (0)}M_{2;1}P^iL_2^{p-2} + St^{(s),(1)}M_{2;1}P^{i-1}L_2^{p-2}\\
&= {[s,0]}(-1)^{k}(k+1)\binom krL_2^{p-2}Q_{2,0}^rQ_{2,1}^{k-r}\\
&\qquad + {[s,1]}(-1)^{k}(k+1)\binom k{r-1}L_2^{p-2}Q_{2,0}^{r-1}Q_{2,1}^{k-r+1}\\
&= (-1)^{k}(k+1)\Big(\binom kr{[s,0]}L_2^{p-2}Q_{2,0}^rQ_{2,1}^{k-r}\\
&\hskip3cm + \binom k{r-1}{[s,1]}L_2^{p-2}Q_{2,0}^{r-1}Q_{2,1}^{k-r+1}\Big).
\end{align*}

Now the theorem follows from this equality and Proposition \ref{md2.9}.
\end{proof}
\begin{thm} Suppose $s,i$ are nonnegative integers. We have
$$St^{(s),(i)}R_{2;0,1} = \begin{cases}  (-1)^{k+1}(k+1)\binom kr R_{2;1}Q_{2,0}^rQ_{2,1}^{k-r},\\ 
\hskip5cm s=0, i=kp+r, 0\leqslant r\leqslant k<p,\\
 (-1)^{k+1}(k+1)\binom krR_{2;0}Q_{2,0}^rQ_{2,1}^{k-r},\\
\hskip5cm s=1, i=kp+r, 0\leqslant r\leqslant k<p,\\
\displaystyle{(-1)^k(k+1)\binom kr\Big(R_{2;1}\sum_{a\in I(1,s)}(-1)^aQ_{2,0}^{pa+1+r}Q_{2,1}^{\frac{p^{s-1}-p}{p-1}-(p+1)a+k-r}}\\ 
\hskip3cm \displaystyle{- R_{2;0}\sum_{a\in I(0,s)}(-1)^aQ_{2,0}^{pa+r}Q_{2,1}^{\frac{p^{s-1}-1}{p-1}-(p+1)a+k-r}\Big)},\\
\hskip5cm s>1, i=kp+r, 0\leqslant r\leqslant k<p,\\
0, \hskip4.6cm \text{otherwise}.
 \end{cases}$$
\end{thm}
\begin{proof} Since $\deg R_{2;0,1} = (p-2)(p+1)+2$, $St^{(s),(i)}R_{2;0,1} =0$ for $i \geqslant p^2$. Suppose $i<p^2$ and $i= kp+r$ with $0\leqslant k,r <p$. Using the Cartan formula and Lemma \ref{bd2.2}, we have
$$St^{(s),(0)}(x_1x_2) = y_1^{p^s}x_2-x_1y_2^{p^s} = -{[1;s]} = (M_{2;1}{[1,s]} - M_{2;0}{[0,s]})/L_2.$$
Since $R_{2;0,1} = x_1x_2L_2^{p-2}$, using the Cartan formula, the above equality and Lemma \ref{bd3.2}, we get
\begin{align*} St^{(s),(i)}R_{2;0.1} &= St^{(s),(0)}(x_1x_2)P^iL_2^{p-2}\\
&= (M_{2;1}{[1,s]} - M_{2;0}{[0,s]})(-1)^k(k+1)\binom kr L_2^{p-3}Q_{2,0}^rQ_{2,1}^{k-r}.
\end{align*}
Combining this equality and Proposition \ref{md2.9}, we get the theorem.
\end{proof}
{}

\bigskip

Department of Mathematics, University of  Quynhon 

170 An Duong Vuong, Quynhon, Vietnam

E-mail address: nguyensum@qnu.edu.vn

\end{document}